\documentclass[11pt]{amsart}
\usepackage{amsmath}
\usepackage{amsfonts}
\usepackage{float}
\usepackage{mathrsfs}
\usepackage{mathtools}
\usepackage{latexsym}
\usepackage{amsmath,amsthm,amssymb, amscd,color}
\usepackage{mathrsfs}%For mathscr
\usepackage[all]{xy}
\usepackage[all]{xy}
\usepackage{graphicx}
\usepackage{subcaption}
\usepackage{enumerate}
\usepackage[shortlabels]{enumitem}

\usepackage{subcaption} % Add the subcaption package

\usepackage{resizegather}
\usepackage{tikz}
\usepackage{tikz-cd}
\usepackage{hyperref}
\usepackage{verbatim}
\usepackage{latexsym}
\usepackage{mathabx}
\usetikzlibrary{decorations.pathmorphing}
\usepackage{anysize}
\marginsize{2.5cm}{2.5cm}{2.5cm}{2.5cm}

\input xy \xyoption{all}
\usepackage{blkarray}  
\usepackage{framed}
\usepackage[utf8]{inputenc}

\usepackage{blkarray} 
\usepackage{framed}

\theoremstyle{plain}
\newtheorem{theorem}{Theorem}[section]
\newtheorem{lemma}[theorem]{Lemma}
\newtheorem{corollary}[theorem]{Corollary}
\numberwithin{equation}{section}

\theoremstyle{definition}
\newtheorem{definition}[theorem]{Definition}
\newtheorem{example}[theorem]{Example}

\newtheorem{proposition}[theorem]{Proposition}
\newtheorem{remark}[theorem]{Remark}

\theoremstyle{remark}

\newcommand{\RR}{\mathbb{R}}

\newcommand{\NN}{\mathbb{N}}

\begin{document}

\title[Invariant circles and phase portraits of cubic vector fields on sphere]{Invariant circles and phase portraits of cubic vector fields on the sphere}
%\author{Joji Benny and Soumen Sarkar}
% \address{Department of Mathematics, Indian Institute of Technology Madras, India}

\author[Joji B.]{Joji Benny}
\address{Department of Mathematics, Indian Institute of Technology Madras, India}
\email{jojikbenny@gmail.com}

\author[S. Jana]{Supriyo Jana}
\address{Department of Mathematics, Indian Institute of Technology Madras, India}
\email{supriyojanawb@gmail.com}

\author[S. Sarkar]{Soumen Sarkar}
\address{Department of Mathematics, Indian Institute of Technology Madras, India}
\email{soumen@iitm.ac.in}

\date{\today}
\thanks{}

\abstract 
In this paper, we characterize and study dynamical properties of cubic vector fields on the sphere $\mathbb{S}^2 = \{(x, y, z) \in \RR^3  ~|~ x^2+y^2+z^2 = 1\}$. We start by classifying all degree three polynomial vector fields on $\mathbb{S}^2$ and determine which of them form Kolmogorov systems. Then, we show that there exist completely integrable cubic vector fields on $\mathbb{S}^2$ and also study the maximum number of various types of invariant circles for homogeneous cubic vector fields on $\mathbb{S}^2$. We find a tight bound in each case. Further, we also discuss phase portraits of certain cubic Kolmogorov vector fields on $\mathbb{S}^2$. 

%number of certain invariant algebraic subsets of $S_{p,q}$ for the vector field $\mathcal{X}$ if either $p>1$ or $q>1$.
\endabstract

\maketitle

\section{Introduction}
Let  $P, Q, R$ be polynomials in $\mathbb{R}[x,y, z]$. Then, the following system of differential equations
\begin{equation} \label{eq: I1}
 \frac{dx}{dt} = P(x, y, z), \frac{dy}{dt} = Q(x, y, z), \frac{dz}{dt} = R(x, y, z)  \end{equation}
is called a polynomial differential system in $\mathbb{R}^3$. The differential operator
\begin{equation}  \label{eq: I2}
 \mathcal{X} = P \frac{\partial}{\partial x} + Q   \frac{\partial}{\partial y} + R \frac{\partial}{\partial z} 
\end{equation}
is called the vector field associated with the system \eqref{eq: I1}. The degree of the polynomial vector field in \eqref{eq: I2} is defined to be $\max \{ \deg (P), \deg (Q), \deg (R) \}$. 

 %When $n=2$ in \eqref{eq: I1}, this differential system has been studied since 1900 possibly because of the second part of the Hilbert sixteenth problem (see \cite{Ila02} and some references therein). 

 The system \eqref{eq: I1} is called a polynomial Kolmogorov system in $\RR^3$ when $P=x P'$, $Q=yQ'$ and $R=zR'$ for some $P', Q', R' \in \RR[x, y, z]$. The associated vector field is called a polynomial Kolmogorov vector field.

 An {\it invariant algebraic set} for \eqref{eq: I2} is a subset $A \subset \mathbb{R}^3$ such that $A$ is the zero set of some $f(x,y,z) \in \mathbb{R}[x,y,z]$ and $\mathcal{X}f = Kf$ for some $K \in \mathbb{R}[x,y,z]$. Here, the polynomial $K$ is called the \textit{cofactor} of $f$. Moreover, $\mathcal{X}$ is also called a vector field on $A$.

The Darboux theory of integrability and its generalizations \cite{LliBol11} posits that if a system in $\RR^2$ has sufficiently many invariant algebraic sets, then it has a rational first integral. This theory can be generalized in $\RR^n$ also, see \cite{DuLlAr06}. Determining invariant algebraic curves for vector fields on surfaces has become a problem of its own; see \cite{Lli08} and the references therein. Perhaps the simplest invariant algebraic set that a vector field on a sphere can have is an invariant circle since a circle on $\mathbb{S}^2$ is always an intersection of a plane with $\mathbb{S}^2$. An interesting problem involving invariant circles is determining when an invariant circle is a limit cycle or at least when it is a periodic orbit. Note that invariant circles for polynomial vector fields on $\mathbb{S}^2$ are discussed in \cite{LliPes06} and \cite{LliZho11} for quadratic vector fields. 

On the other hand, the Kolmogorov system in $\RR^4_{+}$ has been studied in \cite{LlXi17}. Then, the integrability of a class of Kolmogorov systems in $\RR^n$ has been explored in \cite{LlRaRa20}. Importantly, Kolmogorov systems \cite{kol36} have found applications in Plasma Physics \cite{LMQ22}, Economics \cite{AFSS94} and other areas, including Ecology. In this paper, we characterize cubic polynomial vector fields and study the dynamics of cubic Kolmogorov vector fields on $\mathbb{S}^2$. We remark that the article \cite{AP96} discusses phase portraits of certain degree 3 polynomial vector fields on the Poincaré sphere.

This paper is organized in the following way. In Section \ref{sec:prelim}, we discuss the transformation of the vector fields under the stereographic projection formulae for projection from the `South pole'. Then we recall the definition of extactic polynomial, first integral and complete integrability. We show that if $\gamma$ is an orbit of $\mathcal{X}$ on $S^2$, then the cone on $\gamma$ is also invariant for $\mathcal{X}$. We recall the equation of such a cone when $\gamma$ is a circle.

%we present some basic definitions and useful results on polynomial vector fields in some Euclidean spaces.

In Section \ref{sec:cubic-vfld}, we give a necessary and sufficient condition when a cubic vector field in $\RR^3$ is a vector field on $\mathbb{S}^2$, see Theorem \ref{thm:cubic-vector-field}  and further also determine when the invariant great circle is a periodic orbit. We characterize cubic Kolmogorov vector fields on $\mathbb{S}^2$ in Corollary \ref{cr:kolm3}. We discuss the existence of an invariant great circle for a cubic vector field on $\mathbb{S}^2$ in Theorem \ref{thm:invariant-great-circle}. We prove that there exists a large class of completely integrable cubic vector fields on $\mathbb{S}^2$, see Theorem \ref{thm:cmplt-int-vfld}.  We prove that there exist homogeneous and non-homogeneous cubic vector fields on $\mathbb{S}^2$ having invariant great circle as a periodic orbit, see Proposition \ref{prop:hom-non_hom-periodic-orbit}.

In Section \ref{sec: 3homo_vf_on_S2}, the focus of our discussion centers around homogeneous cubic vector fields on $\mathbb{S}^2$.  In Theorem \ref{thm:inv-gcircle-one-zero}, we find tight bounds on the number of invariant great circles that a cubic homogeneous vector field on $\mathbb{S}^2$ can have. We also exhibit a condition in Proposition \ref{prop:inv_not_gcycle} when a cubic homogeneous vector field on $\mathbb{S}^2$ has an invariant circle which is not a great circle.

%In Section \ref{sec:cubic-kolm}, we draw phase portraits of cubic Kolmogorov systems defined on $\mathbb{S}^2$ and discuss their dynamics.

 In Section \ref{sec:cubic-kolm}, we study cubic Kolmogorov systems on $\mathbb{S}^2$ in detail. In particular, we draw phase portraits for various possible cubic Kolmogorov systems and show that a broad class of such systems does not admit a periodic orbit in Theorem \ref{thm:periodic_orbit_kolm}. We give a sufficient condition when a singular point in $\mathbb{S}^2 \setminus \{z=0\}$ of a cubic Kolmogorov vector field on $\mathbb{S}^2$ is either center or focus, see Theorem \ref{thm:center-focus}.

 %The second topic that we study in this paper is degree three (cubic) Kolmogorov systems on $\mathbb{R}^3$ that restricts to a vector field on $\mathbb{S}^2$.Volterra's objective was to explain the oscillatory levels of fish catches in the Adriatic sea, see Chapter 3 in \cite{Mur02}. In addition to its origins in predator-prey dynamics, Lotka-Volterra systems and their generalization,  In this paper, we shall look at cubic Kolmogorov systems defined on $\mathbb{R}^3$ which restrict to a vector field on $\mathbb{S}^2.$ 

%%%%%%%%%%%%%%%%%%%%%%%%%%%
%%%%%%%%%%%%%%%%%%%%%%%%%%%
%%%%%%%%%%%%%%%%%%%%%%%%%%%

\section{Preliminaries}    \label{sec:prelim}

% Let $S$ be a hypersurface in $\mathbb{R}^n$  defined by the zeroes of a non-constant polynomial $h \in \mathbb{R}[x,y, \ldots,x_n]$. We say that a vector field $\mathcal{X}$ of the form \eqref{eq: I2} is defined on $S$ if $(R_1,R_2,\ldots,R_n)\cdot \nabla h=0$ for all points on the hypersurface $S$. This is equivalent to saying that $\mathcal{X}h= Kh$ for some polynomial $K \in \mathbb{R}[x,y, \ldots,x_n]$.  Because  $h(a_1,a_2,...,a_n)=0$ for all points $(a_1,a_2,...,a_n)$ on the hypersurface S. The hypersurface  $S$ is called a regular hypersurface if   $\nabla h \neq 0$ for all points on $S$. This hypersurface is called irreducible if $h$ is irreducible. The \textit{degree} of the hypersurface $S$ is defined to be the degree of $h$.%\bigskip

In this section, we discuss the transformation of the vector fields under the stereographic projection formulae for projection from the `South pole'.  Note that these have been obtained in \cite{LliZho11}.  Then we recall the concept of extactic polynomial which helps to find invariant hyper-surfaces. We recall the definition of first integral and complete integrability. We show that if $\gamma$ is an orbit of $\mathcal{X}$ on $S^2$, then the cone on $\gamma$ is also invariant for $\mathcal{X}$. We recall the equation of such a cone when $\gamma$ is a circle. Recall that the map
$$ \Phi \colon \mathbb{S}^2 - \{ (0, 0, -1)  \}  \longrightarrow \mathbb{R}^2  \quad \mbox{defined by } (x, y, z) \mapsto (u, v) $$
where,
$$ u = \frac{x}{1+z} \quad, \quad v = \frac{y}{1+z}$$ is called the stereographic projection from the South pole. 
For a vector field $\mathcal{X} = (P, Q, R)$ given by
\begin{equation*}
   % \begin{split}
        P = \sum_{i+j+k \leq n} p_{ijk} x^i y^j z^k,   \quad
        Q = \sum_{i+j+k \leq n} q_{ijk} x^i y^j z^k,   \quad
        R = \sum_{i+j+k \leq n} r_{ijk} x^i y^j z^k,   
    %\end{split}
\end{equation*}
the induced map $\Phi_* (P, Q, R) = (\mathcal{P}, \mathcal{Q})$ is determined by
\begin{equation*}
    \begin{split}
        \dot{u} = \mathcal{P} = \Tilde{P}(u,v) - u \Tilde{R}(u,v),  \\
        \dot{v} = \mathcal{Q} = \Tilde{Q}(u,v) - v \Tilde{R}(u,v),
    \end{split}
\end{equation*}
where,
\begin{equation*}
    \begin{split}
        \Tilde{P}(u,v) = \sum_{i+j+k \leq n} p_{ijk} (2u)^i (2v)^j (1-u^2-v^2)^k (u^2+v^2+1)^{n-i-j-k}, \\
        \Tilde{Q}(u,v) = \sum_{i+j+k \leq n} q_{ijk} (2u)^i (2v)^j (1-u^2-v^2)^k (u^2+v^2+1)^{n-i-j-k}, \\
        \Tilde{R}(u,v) = \sum_{i+j+k \leq n} r_{ijk} (2u)^i (2v)^j (1-u^2-v^2)^k (u^2+v^2+1)^{n-i-j-k}.
    \end{split}
\end{equation*}

One of the best tools in order to search for invariant algebraic hypersurfaces is the following. Let
$W$ be a vector subspace of $\RR[x_1,...,x_n]$ generated by the independent
polynomials $v_1 ,... , v_l$, i.e., $W=\langle v_1,...,v_{\ell} \rangle$. The extactic polynomial of $\mathcal{X}$ associated with $W$ is the polynomial
$$\mathcal{E}_W(\mathcal{X})= 
\begin{vmatrix}
v_1 & \dots &v_{\ell}\\
\mathcal{X} (v_1) &\dots &\mathcal{X} (v_{\ell})\\
\vdots &\vdots &\vdots\\
\mathcal{X}^{{\ell}-1}(v_1) &\dots &\mathcal{X}^{{\ell}-1} (v_{\ell})
\end{vmatrix}
$$
where $\mathcal{X}^j (v_i) = \mathcal{X}^{j-1}( \mathcal{X} (v_i))$ for all $i, j$. From the properties of determinants, it follows that the definition of the extactic polynomial is independent of the chosen
basis of $W$.

\begin{proposition}\label{extactic-polynomial}
 \cite[Proposition 1]{LlMe07}  Let $\mathcal{X}$ be a polynomial vector field in $\mathbb{R}^n$ and $W$
a finite dimensional vector subspace of $\mathbb{R}[x_1, x_2, \ldots , x_n]$ with $\dim(W) > 1$. If $f=0$ is an invariant algebraic hypersurface for the vector field $\mathcal{X}$ with $f\in W$, then $f$ is a factor of $\mathcal{E}_W(\mathcal{X})$.
\end{proposition}
\par The multiplicity of an invariant algebraic hypersurface $f = 0$ with $f\in W$, is the largest positive integer
$k$ such that $f^k$ divides the extactic polynomial $\mathcal{E}_W(\mathcal{X})$ when $\mathcal{E}_W(\mathcal{X})\neq 0$, otherwise the multiplicity
is infinite. For more details on this multiplicity, see \cite{ChLlPe07, LlZh09}.
\begin{definition}
Let $U$ be an open subset of $\RR^3$. A non-constant analytic map $H \colon U \to \mathbb{R}$ is called a first integral
of the vector field \eqref{eq: I2} on $U$ if $H$ is constant on all solution curves of the system \eqref{eq: I1} contained in $U$ ; i.e., $H(x(t),y(t), z(t)) =$ constant for all values of $t$ for which the solution $(x(t),y(t), z(t))$ is defined and contained in $U$.
\end{definition}
Note that $H$ is
a first integral of the vector field \eqref{eq: I2} on $U$ if and only if $\mathcal{X} H=0$ on $U$.

\begin{definition}
    The system \eqref{eq: I1} is called completely integrable in an open set $U$ if it has 2 independent first integrals.
\end{definition}
If the system \eqref{eq: I1} is completely integrable with 2 independent first integrals $H_1$ and $H_2$ then the orbits of the system are contained in $\{H_1=c_1\}\cap \{H_2=c_2\}$ for some $c_1,c_2\in \RR$. In \cite{Ll2021}, the complete integrability of vector fields in $\RR^n$ is discussed.

Recall that the intersection of a plane $ax+by+cz+d=0$ with $\mathbb{S}^2$ is a circle. We can choose $a,b,c,d$ such that $a^2 + b^2 + c^2 = 1$ and $|d| < 1.$ Any circle on $\mathbb{S}^2$ can be obtained in this way. If the plane passes through the origin, the intersection is called a great circle. Note that a vector field on $\mathbb{S}^2$ is invariant by $SO(3)$\footnote{$SO(3)$ is the group of all rotations about the origin in $\RR^3$.}. Hence, if a polynomial vector field on $\mathbb{S}^2$ has an invariant circle $\{ax+by+cz+d=0\}\cap \mathbb{S}^2$, then we can assume the circle to be $\{z+d=0\}\cap \mathbb{S}^2$ with $|d| <1$.

We now state a result for homogeneous polynomial vector fields on $\mathbb{S}^2$, which is proved for degree two in \cite{LliPes06}, and an entirely similar proof can be given for homogeneous vector fields of any degree.

\begin{proposition} \label{prop: 33}
Let $\gamma = \{ \phi(t) | t \in \mathbb{R} \} \subset \mathbb{S}^2$ be an orbit of $\mathcal{X}.$ If $\mathcal{X}$ is a homogeneous polynomial vector field on $\mathbb{S}^2,$ then $\mathcal{X}$ is tangent to the surface
$S(\gamma) = \{sp ~|~ s \in \mathbb{R}, p \in \gamma       \}.$

\end{proposition}

We want to look at the circles on $\mathbb{S}^2$ which are invariant with respect to the flow of the vector field $\mathcal{X}$ on $S^2$. In this case, Proposition \ref{prop: 33} will imply that the entire cone on the circle is also invariant with respect to the flow of $\mathcal{X}$ if the circle is invariant by the flow of $\mathcal{X}.$
In \cite{LliPes06}, the authors describe the equation of the cone of such a circle as 
\begin{equation}    \label{eq: 34}
 (a^2 - d^2)x^2 + (b^2 - d^2)y^2 + (c^2 - d^2)z^2 + 2abxy + 2acxz + 2bcyz=0  
\end{equation}
if $d \neq 0.$ If $d = 0,$ then the cone itself becomes the plane given by $ax + by + cz = 0.$
When $d=0,$ the intersection of $\{ ax + by + cz = 0    \}$ with $\mathbb{S}^2$ is a great circle.

%=============================
%=============================
%=============================

\section{Cubic vector fields on $\mathbb{S}^2$}        \label{sec:cubic-vfld}
In this section, we characterize cubic vector fields on the standard sphere in $\RR^3$. Then we classify cubic Kolmogorov vector fields on $S^2$. We also study those cubic vector fields which have an invariant great circle.
\begin{lemma}\cite[Lemma 4.1]{JaSa23}\label{lem:32}
    Let $n\in \NN$ and $Q_1,Q_2,Q_3$ are polynomials in $\RR[x,y,z]$ such that the polynomial $Q_1x^n+Q_2y^n+Q_3z^n$ is zero. Then $Q_1=Ay^n+Bz^n$, ${Q_2=-Ax^n+Cz^n}$ and $Q_3=-Bx^n-Cy^n$ for some polynomials $A,B,C$.
\end{lemma}
%\begin{proof}
 %   The proof is similar to the proof of \cite[Lemma 4.1]{JaSa23}.
%\end{proof}

\begin{theorem}\label{thm:cubic-vector-field}
    Let $\mathcal{X}=(P,Q,R)$ be a cubic polynomial vector field in $\RR^3$. Then $\mathcal{X}$ is cubic vector field on $\mathbb{S}^2$ if and only if there exist $f,g,h,A,B,C\in \RR[x,y,z]$ such that
    \begin{equation}\label{eq:cubic-form}
        \begin{split}
            P&=(1-x^2-y^2-z^2)f+Ay+Bz,\\
            Q&=(1-x^2-y^2-z^2)g-Ax+Cz,\quad \text{and}\\
            R&=(1-x^2-y^2-z^2)h-Bx-Cy
        \end{split}
    \end{equation}
where $f,g,h$ are linear polynomials and $A,B,C$ are quadratic polynomials without any constant term. Moreover, this vector field has cofactor $-2(f x+g y+hz)$ for $\mathbb{S}^2$.
\end{theorem}
\begin{proof}
 We write $P=P^{(3)}+P^{(2)}+P^{(1)}+P^{(0)}$ where $P^{(j)}$ is the degree $j$ homogeneous part of $P$. Similarly, we write $Q =Q^{(3)}+Q^{(2)}+Q^{(1)}+P^{(0)}$ and $R = R^{(3)}+R^{(2)}+R^{(1)}+R^{(0)}$ also in this fashion.

    Suppose $\mathcal{X}=(P, Q, R)$ is a vector field on $\mathbb{S}^2$. Then, it must satisfy
    \begin{equation}\label{cubic}
    2(Px+Qy+Rz)=K(x^2+y^2+z^2-1).
        \end{equation} for some $K\in \RR[x,y,z]$ with $\deg(K)\leq 2$.
Assume that $K^{(j)}$ is the $j$ degree homogeneous part of $K$. Then $K^{(0)}$ is $0$ since there is no constant term on the left side of \eqref{cubic}. Now, Comparing the degree 4, degree 3, degree 2 and degree 1 terms in \eqref{cubic}, we get 
    \begin{equation*}
        \begin{split}
            2(P^{(3)}x+Q^{(3)}y+R^{(3)}z)&=K^{(2)}(x^2+y^2+z^2),\\
           2(P^{(2)}x+Q^{(2)}y+R^{(2)}z)&=K^{(1)}(x^2+y^2+z^2),\\
            2(P^{(1)}x+Q^{(1)}y+R^{(1)}z)&=-K^{(2)},\\
            2(P^{(0)}x+Q^{(0)}y+R^{(0)}z)&=-K^{(1)},
        \end{split}
    \end{equation*}
    respectively.
    Hence $$P^{(3)}x+Q^{(3)}y+R^{(3)}z=-(P^{(1)}x+Q^{(1)}y+R^{(1)}z)(x^2+y^2+z^2), \quad \text{and}$$  $$P^{(2)}x+Q^{(2)}y+R^{(2)}z=-(P^{(0)}x+Q^{(0)}y+R^{(0)}z)(x^2+y^2+z^2).$$ By Lemma \ref{lem:32},
    \begin{equation*}
        \begin{split}
     P^{(3)}&=-P^{(1)}(x^2+y^2+z^2)+A_1y+B_1z,\\     Q^{(3)}&=-Q^{(1)}(x^2+y^2+z^2)-A_1x+C_1z,\\
    R^{(3)}&=-R^{(1)}(x^2+y^2+z^2)-B_1x-C_1y
        \end{split}
    \end{equation*}
where each of $A_1,B_1,C_1$ is either 0 or a quadratic homogeneous polynomial in $\RR[x, y, z]$, and
        \begin{equation*}
        \begin{split}
    P^{(2)}&=-P^{(0)}(x^2+y^2+z^2)+A_2y+B_2z\\
    Q^{(2)}&=-Q^{(0)}(x^2+y^2+z^2)-A_2x+C_2z\\  R^{(2)}&=-R^{(0)}(x^2+y^2+z^2)-B_2x-C_2y.
   \end{split}
    \end{equation*}
where each of $A_2, B_2, C_2$ is either 0 or a linear homogeneous polynomial in $\RR[x, y, z]$.
Hence
        \begin{equation*}
%  \begin{split}
P =P^{(3)}+P^{(2)}+P^{(1)}+P^{(0)} =(1-x^2-y^2-z^2)(P^{(1)}+P^{(0)})+(A_1+A_2)y+(B_1+B_2)z.
 %       \end{split}
    \end{equation*}
Denoting $f := P^{(1)}+P^{(0)}$, $A:=A_1+A_2$ and $B:=B_1+B_2$; $P$ becomes $(1-x^2-y^2-z^2)f+Ay+Bz$. Similarly, we get $Q$ and $R$ as in the form in \eqref{eq:cubic-form}.

Note that if $P,Q,R$ are given by \eqref{eq:cubic-form}, then they satisfy \eqref{cubic} with ${K=-2(fx+gy+hz)}$. Hence, $\mathcal{X}=(P,Q,R)$ is a vector field on $\mathbb{S}^2$. Thus, the converse part is true.
\end{proof}

\begin{corollary}      \label{cr:kolm3}
    Suppose $\mathcal{X}=(P,Q,R)$ is the cubic Kolmogorov vector field in $\RR^3$. Then $\mathcal{X}$ is a vector field on $\mathbb{S}^2$ if and only if there exist $\alpha,\beta,\gamma,a,b,c\in \RR$ such that
    \begin{equation}\label{cubic-kolm-form}
        \begin{split}
            P&=x(\alpha(1-x^2-y^2-z^2)+ay^2+bz^2),\\
            Q&=y(\beta(1-x^2-y^2-z^2)-ax^2+cz^2), \quad \text{and}\\
            R&=z(\gamma(1-x^2-y^2-z^2)-bx^2-cy^2).
        \end{split}
    \end{equation}
\end{corollary}
\begin{proof}
% $\mathcal{X}=(P,Q,R)$ is the vector field associated with a cubic Kolmogorov system in $\RR^3$. 
Suppose $\mathcal{X}=(P, Q, R)$ is a vector field on $\mathbb{S}^2$. Then by Theorem \ref{thm:cubic-vector-field}, $P,Q,R$ are given by \eqref{eq:cubic-form}. Since $\mathcal{X}$ is Kolmogorov, $x,y,z$ divide $P, Q, R$, respectively. Suppose $$P'x=P=(1-x^2-y^2-z^2)f+Ay+Bz.$$ Then, we get $P'(0,0,0)x=f$ by comparing the linear part on both sides. Assume that ${P'(0,0,0)=\alpha}$, we get $f=\alpha x$. Similarly, for $Q'y=Q$ and $R'z=R$, one gets $g=\beta y$ and $h=\gamma z$ for some $\beta, \gamma \in \RR$. Therefore, $x$ divides $Ay+Bz$, $y$ divides $-Ax+Cz$, and $z$ divides $-Bx-Cy$. Suppose $$P_1x=Ay+Bz,\quad Q_1y=-Ax+Cz \quad \text{and}\quad R_1z=-Bx-Cy$$ for some polynomials $P_1, Q_1,$ and $R_1$ with $\deg(P_1), \deg(Q_1), \deg(R_1) \leq 2$. Then, we get that $P_1x^2+Q_1y^2+R_1z^2=x(Ay+Bz)+y(-Ax+Cz)+z(-Bx-Cy)=0$. So, by Lemma \ref{lem:32}, $P_1=ay^2+bz^2,  Q_1=-ax^2+cz^2$ and $R_1=-bx^2-cy^2$ for some $a, b, c \in \RR$. Thus, we get $P,Q,R$ as in \eqref{cubic-kolm-form}.

Suppose $P,Q,R$ are given by \eqref{cubic-kolm-form}. Then, the associated differential system of the vector field ${\mathcal{X}=(P,Q,R)}$ is a cubic Kolmogorov vector field. Also, by Theorem \ref{thm:cubic-vector-field}, $\mathcal{X}$ is a vector field on $\mathbb{S}^2$. Thus, the converse part is true.
\end{proof}

\begin{theorem}\label{thm:cmplt-int-vfld}
There exist completely integrable cubic vector fields on $\mathbb{S}^2$.
\end{theorem}
\begin{proof}
%     We choose linear polynomials $f,g,h\in \RR[x,y,z]$, and $a,b,c\in \RR$ not all zero such that
%     \begin{enumerate}[(i)]
%         \item $af+bg+ch=0$ and
%         \item $fx+gy+hz=0$.
%     \end{enumerate}

% We show that there exists $a,b,c\in \RR$ and $f,g,h\in \RR[x,y,z]$ which satisfy (i) and (ii). Choose $a\in \RR\backslash\{0\}$. By Lemma \ref{lem:32}, $fx+gy+hz=0$ if and only if $f=\alpha y+\beta z$, $g=-\alpha x+\gamma z$ and ${h=-\beta x-\gamma y}$ for some $\alpha, \beta, \gamma\in \RR$. If $af+bg+ch=0$ then $a(\alpha y+\beta z)+b(-\alpha x+\gamma z)+c(-\beta x-\gamma y)=0$, which implies that $(-b\alpha-c\beta)x+(a\alpha-c\gamma)y+(a\beta+b\gamma)z=0$. So, $-b\alpha-c\beta=0$, $a\alpha-c\gamma=0$, and $a\beta+b\gamma=0$. Hence, $\alpha=\frac{c\gamma}{a}$ and $\beta=-\frac{b\gamma}{a}$. So, $f=\frac{c\gamma}{a}y-\frac{b\gamma}{a}z$, $g=-\frac{c\gamma}{a} x+\gamma z$ and $h=\frac{b\gamma}{a}x-\gamma y$ with $\gamma\neq 0$ will satisfy both (i) and (ii).

% $$$$

 By Lemma \ref{lem:32}, $f'x+g'y+h'z=0$ for some $f', g', h' \in \RR[x,y,z]$ if and only if $f'=\alpha y+\beta z$, $g'=-\alpha x+\gamma z$ and ${h'=-\beta x-\gamma y}$ for some $\alpha, \beta, \gamma\in \RR$. Suppose $\gamma \neq 0$ and $a \in \RR \setminus \{0\}$. Then there exist 
 $b, c \in \RR$ such that $\alpha=\frac{c\gamma}{a}$ and $\beta=-\frac{b\gamma}{a}$. Define $f:=\frac{c\gamma}{a}y-\frac{b\gamma}{a}z$, $g:=-\frac{c\gamma}{a} x+\gamma z$ and $h:=\frac{b\gamma}{a}x-\gamma y$. Then $$af+bg+ch=0 \quad \mbox{and} \quad fx+gy+hz=0.$$
Now, we choose $A=\frac{c}{a}C$ and $B=-\frac{b}{a}C$ for any quadratic polynomial $C$ with $C(0,0,0)=0$. 
% Let $f' = \alpha y + \beta z,$ $g' = -\alpha x + \gamma z,$ $h' = -\beta x - \gamma y,$ and $a \neq 0$ where $\alpha, \beta, \gamma \in \RR \setminus \{0\}$. So there exists $b,c \in \RR$ such that $a \alpha = c \gamma$ and $a \beta = -b \gamma.$ With these $f', g', h'$ and $A, B, C$ we define
Consider the vector field $\mathcal{X} = (P,Q,R)$ given by 
    \begin{equation*}
        \begin{split}
            P&=(1-x^2-y^2-z^2)f+Ay+Bz,\\
            Q&=(1-x^2-y^2-z^2)g-Ax+Cz,\quad \text{and}\\
            R&=(1-x^2-y^2-z^2)h-Bx-Cy.
        \end{split}
    \end{equation*}
By Theorem \ref{thm:cubic-vector-field}, $\mathcal{X}$ is a vector field on $\mathbb{S}^2$ with cofactor $-2(fx+gy+hz)$. Hence, $x^2+y^2+z^2-1$ is a first integral of $\mathcal{X}$. Again,
\begin{equation*}
    \begin{split}
        \mathcal{X}(ax+by+cz)=&aP+bQ+cR\\
                    =&(1-x^2-y^2-z^2)(af+bg+ch)+a(Ay+Bz)+b(-Ax+Cz)+c(-Bx-Cy)\\
                    =&(-bA-cB)x+(aA-cC)y+(aB+bC)z=0.
    \end{split}
\end{equation*}
Hence, $\mathcal{X}$ has two independent first integrals $x^2+y^2+z^2-1$ and $ax+by+cz$, which makes $\mathcal{X}$ completely integrable in $\RR^3$.
\end{proof}
\begin{corollary}
 If one of $a, b, c$ is non-zero, then there exists a cubic vector field on $\mathbb{S}^2$ which has a  first integral $ax+by+cz$.
\end{corollary}

The rest of this section characterizes vector fields on $\mathbb{S}^2$ which have invariant great circles. Suppose $P,Q,R$ are given by \eqref{eq:cubic-form}, where
\begin{equation}\label{eq:p,q,r}
%\begin{split}
    A=\sum\limits_{1\leq i+j+k \leq 2}a_{ijk}x^iy^jz^k, ~~
    B=\sum\limits_{1\leq i+j+k \leq 2}b_{ijk}x^iy^jz^k, ~ \mbox{and} ~
    C=\sum\limits_{1\leq i+j+k \leq 2}c_{ijk}x^iy^jz^k.   
 %   \end{split}
\end{equation}
Then, under the stereographic projection, the system given by \eqref{eq:cubic-form} becomes
\begin{equation*}
    \begin{split}
        \dot u&=\Tilde P(u,v)-u\Tilde R(u,v)=\mathcal{P}(u,v),\\
        \dot v&=\Tilde Q(u,v)-v\Tilde R(u,v)=\mathcal{Q}(u,v),\text{ where }
\end{split}
\end{equation*}
\begin{equation*}
    \begin{split}
        \Tilde{P}=&((u^2+v^2+1)^2-(2u)^2-(2v)^2-(1-u^2-v^2)^2)\Tilde f+\Tilde A (2v)+\Tilde B(1-u^2-v^2)\\
        =&2\Tilde Av+\Tilde B(1-u^2-v^2).
    \end{split}
\end{equation*}

Similarly, $\Tilde{Q}=-2\Tilde{A}u+\Tilde{C}(1-u^2-v^2)$ and $\Tilde{R}=-2\Tilde{B}u-2\Tilde{C}v$ where
\begin{equation*}
    \begin{split}
        \Tilde{A}=\sum\limits_{1\leq i+j+k\leq 2} a_{ijk}(2u)^i(2v)^j(1-u^2-v^2)^k(u^2+v^2+1)^{2-i-j-k},\\
        \Tilde{B}=\sum\limits_{1\leq i+j+k\leq 2} b_{ijk}(2u)^i(2v)^j(1-u^2-v^2)^k(u^2+v^2+1)^{2-i-j-k},\\
        \Tilde{C}=\sum\limits_{1\leq i+j+k\leq 2} c_{ijk}(2u)^i(2v)^j(1-u^2-v^2)^k(u^2+v^2+1)^{2-i-j-k}.
    \end{split}
\end{equation*}

\begin{theorem}\label{thm:invariant-great-circle}
    Let $\mathcal{X}=(P,Q,R)$ be a cubic vector field on $\mathbb{S}^2$. Assume $\mathcal{X}$ has an invariant great  circle. Without loss of generality, we can assume that it is $\mathbb{S}^1=\{z=0\}\cap \mathbb{S}^2$. Then, the following hold.
    \begin{enumerate}[(a)]
  \item The vector field $\mathcal{X}$ can be written as \eqref{eq:cubic-form} with
    $$B=b_{020}y^2+b_{110}xy+b_{010}y+B^{'}z, C=-b_{110}x^2-b_{020}xy-b_{010}x+C^{'}z$$ where $B^{'},C^{'}$ are linear polynomials in $\RR[x, y, z]$.
    \item The great circle $\mathbb{S}^1$ is a periodic orbit of $\mathcal{X}$ if and only if the hypersurfaces \\ $\left\lbrace \sum\limits_{1\leq i+j\leq 2}a_{ij0}u^iv^j=0   \right\rbrace$ and $\{u^2+v^2=1\}$ do not intersect in $\RR^2$.
    \end{enumerate}
\end{theorem}
\begin{proof}
(a)
If $\mathcal{X}=(P,Q,R)$ given by \eqref{eq:cubic-form} with $A,B,C$ as in \eqref{eq:p,q,r} is a vector field on $\mathbb{S}^2$ then
$$xP(x,y,z)+yQ(x,y,z)+zR(x,y,z)=0 \text{ for all $(x,y,z)\in \mathbb{S}^2$}.$$ Hence, after the stereographic projection, we get
$$2u\Tilde{P}(u,v)+2v\Tilde{Q}(u,v)+(1-u^2-v^2)\Tilde{R}(u,v)=0 \text{ for all $(u,v)\in \RR^2$}.$$ So, by the above,
\begin{equation}\label{eq:uu.+vv.}
    u\dot{u}+v\dot{v}=u\Tilde{P}(u,v)+v\Tilde{Q}(u,v)-(u^2+v^2)\Tilde{R}(u,v)=-\frac{1}{2}(u^2+v^2+1)\Tilde{R}(u,v).
\end{equation}
Under the stereographic projection $\{z=0\}\cap \mathbb{S}^2$ reduces to unit circle $u^2+v^2=1$ in $\RR^2$. Due to \eqref{eq:uu.+vv.}, one can check that $u^2+v^2-1=0$ is invariant of the vector field $(\mathcal{P},\mathcal{Q})$ if and only if $u^2+v^2-1$ divides $\Tilde{R}=-2(\Tilde{B}u+\Tilde{C}v)$. Again, $u^2+v^2-1$ divides $\Tilde{R}$ if and only if $\Tilde R|_{u^2+v^2=1}=0$. Note that
\begin{equation*}
    \begin{split}
        \Tilde R|_{u^2+v^2=1}=&-2(u \Tilde{B}|_{u^2+v^2=1}+v\Tilde{C}|_{u^2+v^2=1})\\
                             =&-8\left(\sum\limits_{1\leq i+j\leq 2} b_{ij0}u^{i+1}v^j+c_{ij0}u^iv^{j+1} \right).
    \end{split}
\end{equation*}
Hence, $\Tilde R|_{u^2+v^2=1}=0$ implies that $$b_{100}=b_{200}=c_{010}=c_{020}=0, c_{100}=-b_{010}, c_{110}=-b_{020}, c_{200}=-b_{110}.$$ Hence, the first claim follows.

(b) It is well known that the invariant great circle $\mathbb{S}^1 \subset \mathbb{S}^2$ is a periodic orbit of $\mathcal{X}$ on $\mathbb{S}^2$ if and only if there is no singular point on it. Observe that, if $\mathbb{S}^1$ is invariant then $\Tilde R|_{u^2+v^2=1}=0$. Hence, $$\mathcal{P}|_{u^2+v^2=1}=2v\Tilde{A}|_{u^2+v^2=1} \text{ and } \mathcal{Q}|_{u^2+v^2=1}=-2u\Tilde{A}|_{u^2+v^2=1}.$$ Note that $\Tilde{A}|_{u^2+v^2=1}=4\sum\limits_{1\leq i+j\leq 2}a_{ij0}u^iv^j$. So, $\mathbb{S}^1$ is a periodic orbit if and only if the following two equations do not have a common solution.
$$ \sum\limits_{1\leq i+j\leq 2}a_{ij0}u^iv^j=0,~ u^2+v^2=1.$$

% Note that $\Tilde{A}|_{u^2+v^2=1}=4a_{200}u^2+4a_{020}v^2+4a_{110}uv+2a_{100}u(u^2+v^2+1)+2a_{010}v(u^2+v^2+1)=4a_{200}u^2+4a_{020}v^2+4a_{110}uv+4a_{100}u+4a_{010}v.$
\end{proof}
Llibre \cite[Theorem 2]{LliZho11} showed that a quadratic homogeneous vector field on $\mathbb{S}^2$ has no great circle which is a periodic orbit. This statement does not hold true for cubic homogeneous vector fields on $\mathbb{S}^2$.
\begin{proposition}\label{prop:hom-non_hom-periodic-orbit}
\begin{enumerate}[(a)]
        \item There exists a homogeneous vector field having a great circle as a periodic orbit.
    \item There exists a non-homogeneous vector field having a great circle as a periodic orbit.
\end{enumerate}
\end{proposition}
\begin{proof}
(a) Consider the vector field $\mathcal{X}$ on $\mathbb{S}^2$ given by \eqref{eq:cubic-form} with $$A=x^2+y^2, B=y^2+xy, C=-x^2-xy\quad \text{and}\quad f=g=h=0 .$$ By (a) of Theorem \ref{thm:invariant-great-circle}, the intersection $\{z=0\}\cap \mathbb{S}^2$ is an invariant great circle for $\mathcal{X}$. Also, by (b) of Theorem \ref{thm:invariant-great-circle}, the great circle is periodic since $u^2+v^2=0$ and $u^2+v^2=1$ do not have a  common solution.

(b) Consider the vector field $\mathcal{X}$ on $\mathbb{S}^2$ given by \eqref{eq:cubic-form} with $$A=x^2+y^2, \quad B=y^2+xy, \quad C=-x^2-xy$$ and at least one of $f,g,h$ is non-zero. By (a) of Theorem \ref{thm:invariant-great-circle}, $\{z=0\}\cap \mathbb{S}^2$ is an invariant great circle for $\mathcal{X}$. Also, by (b) of Theorem \ref{thm:invariant-great-circle}, the great circle is periodic since $u^2+v^2=0$ and $u^2+v^2=1$ do not have a  common solution.
\end{proof}

\section{Cubic homogeneous vector fields on $\mathbb{S}^2$ } \label{sec: 3homo_vf_on_S2}

%The standard two sphere $\mathbb{S}^2$ can be embedded in $\mathbb{R}^3$ as the zero set of the polynomial ${f = x^2 + y^2 + z^2 - 1}$.
In this section, we study some properties of the vector fields
$$ \mathcal{X} := P \frac{\partial}{\partial x} + Q \frac{\partial}{\partial y} + R \frac{\partial}{\partial z},$$ defined on $\mathbb{S}^2 = \{(x, y,z) \in \RR ~|~  f = x^2 + y^2 + z^2 - 1=0\}$ such that $P,Q,R$ are homogeneous cubic polynomials in $\RR[x, y,z]$. Hence
\begin{equation}  \label{eq: 31}
    \mathcal{X}(f) = K_f (f)
\end{equation}
where $K_f$ is the cofactor.
We see that 
$$\mathcal{X}(f) = 2Px + 2Qy + 2Rz.$$
So, $\mathcal{X}$ is homogeneous implies that $\mathcal{X}(f)$ is also homogeneous, but we see that the right-hand side of \eqref{eq: 31} is not homogeneous unless $K_f = 0.$ Therefore for a homogeneous vector field $\mathcal{X}=(P,Q,R)$ on $\mathbb{S}^2,$ we have
$$ Px+Qy+Rz = 0.$$
Hence by Lemma \ref{lem:32}, the associated differential system of a homogeneous vector field $\mathcal{X}=(P,Q,R)$ on $\mathbb{S}^2$ will be of the following form.
\begin{equation}\label{eq:cubic-hom}
        P=Ay+Bz ,\quad
        Q=-Ax+Cz,\quad \mbox{and} \quad
        R=-Bx-Cy
\end{equation}
for some polynomials $A,B,C$ in $R[x,y,z]$. 
% We write
% $$~~A = a_1x^2 + a_2y^2 + a_3z^2 + a_4xy + a_5xz + a_6yz,$$
% $$B = b_1x^2 + b_2y^2 + b_3z^2 + b_4xy + b_5xz + b_6yz,$$
% $$C = c_1x^2 + c_2y^2 + c_3z^2 + c_4xy + c_5xz + c_6yz.$$

%\subsection{Invariant Circles on $\mathbb{S}^2$ for homogeneous cubic polynomial vector fields}

Next, we discuss invariant circles on $\mathbb{S}^2$ for homogeneous cubic vector fields. 
We see from Proposition \ref{prop: 33} that a circle on $\mathbb{S}^2$ is invariant with respect to the flow of $\mathcal{X}$ if and only if the cone of the circle is invariant with respect to the flow of $\mathcal{X}.$ This implies that the circle given by the intersection of $\{ ax + by + cz + d = 0    \}$ with $\mathbb{S}^2$ is invariant if and only if

\begin{equation}   \label{eq: 35}
\begin{split}
    \mathcal{X} \left( (a^2 - d^2)x^2  + (b^2 - d^2)y^2 + (c^2 - d^2)z^2 + 2abxy + 2acxz + 2bcyz \right) = \\ K \left( (a^2 - d^2)x^2  + (b^2 - d^2)y^2 + (c^2 - d^2)z^2 + 2abxy + 2acxz + 2bcyz \right)  \end{split}
\end{equation}
for some polynomial $K \in \RR[x, y, z]$ where $a^2+b^2+c^2=1$ and $|d| < 1$. 
We remark that 
The cones of the invariant circles $\{ax+by+cz+d=0\}\cap \mathbb{S}^2$ and $\{ax+by+cz-d=0\}\cap \mathbb{S}^2$ are same. Hence, invariant circles which are not great circles always occur in pairs.

In the case of a great circle $\{ ax + by + cz = 0  \} \cap \mathbb{S}^2,$ the cone is the plane $\{  ax + by + cz = 0   \}$ itself, and then the great circle is invariant with respect to the flow of $\mathcal{X}$ if and only if
\begin{equation}  \label{eq: 36}
\mathcal{X} (ax + by + cz) = K' (ax + by + cz)
\end{equation}
for some polynomial $K' \in \RR[x, y, z].$

\begin{theorem}\label{thm:inv-gcircle-one-zero}
Let $\mathcal{X}=(P,Q,R)$ be a cubic homogeneous vector field on $\mathbb{S}^2$. Assume $\mathcal{X}$ has an invariant great  circle. Without loss of generality, we can assume that it is $
\mathbb{S}^1=\{z=0\}\cap \mathbb{S}^2$. If $\mathcal{X}$ has finitely many invariant great circles, then the maximum number of invariant great circles of the form $\{ax+by+cz=0\}\cap \mathbb{S}^2$ is
\begin{enumerate}
    \item 3 if $a=0$ or $b=0$,
   % \item 3 if $b=0$,
    \item 2 if $c=0$,
    \item 2 if $a,b,c$ are non-zero.
\end{enumerate}
Moreover, in each of the above cases, the bound can be reached.
\end{theorem}
\begin{proof}
The vector field $\mathcal{X}$ is given by \eqref{eq:cubic-hom}. Since $z=0$ is invariant, $z$ divides $R=-Bx-Cy$. Hence, by Lemma \ref{lem:32}, $B=L y+B'z$ and $C=-Lx+C'z$ for some  $L,B',C' \in \RR[x,y,z]$ of degree less than or equal to 1. Then
\begin{equation*}
    \begin{split}
        P&=Ay+(Ly+B'z)z=(A+Lz)y+B'z^2=A'y+B'z^2,\\
        Q&=-Ax+(-Lx+C'z)z=-(A+Lz)x+C'z^2=-A'x+C'z^2,\\
        R&=-(Ly+B'z)x-(-Lx+C'z)y=-z(B'x+C'y).
    \end{split}
\end{equation*} where $A':=A+Lz$ is a homogeneous polynomial of degree less than or equal to 2.
We write
\begin{equation*}
\begin{split}
    A'&=a_1x^2+a_2y^2+a_3z^2+a_4xy+a_5xz+a_6yz,\\
    B'&=b_1x+b_2y+b_3z, \mbox{ and} \\
    C'&=c_1x+c_2y+c_3z.
\end{split}
\end{equation*}
(1) The extactic polynomial of $\mathcal{X}$ associated with $\langle y,z \rangle$ is 
$$\mathcal{E}_{\langle y,z\rangle}(\mathcal{X})=\begin{vmatrix}
        y&z\\
        Q&R
    \end{vmatrix}=z(-C'(y^2+z^2)+x(A'-B'y)).$$ Suppose $\mathcal{E}_{\langle y,z\rangle}(\mathcal{X})$ has $\ell$ factors of the form $by+cz$, say $b_iy+c_iz$ for $i=1,2,3,4$. Then $$z(-C'(y^2+z^2)+x(A'-B'y))=\prod\limits_{i=1}^4 (b_iy+c_iz).$$ For $x=0$, the above equation becomes $-zC'(0,y,z)(y^2+z^2)=\prod\limits_{i=1}^4 (b_iy+c_iz)$. But this is not possible since $y^2+z^2$ has no linear factor over $\RR$. So, $\mathcal{E}_{\langle y,z\rangle}(\mathcal{X})$ has at most 3 factors of the form $by+cz$. Hence, for $a = 0$ by the Proposition \ref{extactic-polynomial}, the maximum number of invariant great circles of the form $\{by+cz=0\}\cap \mathbb{S}^2$ is 3.

    To show that the bound can be reached, consider $A'=\prod\limits_{i=1}^2(b_iy+c_iz)+B'y$, $C'=0$ for any linear homogeneous polynomial $B'$.  Then the vector field
$$\mathcal{X} =(P,Q,R)= \left( y\prod\limits_{i=1}^2(b_iy+c_iz)+B'(y^2+z^2), -x\prod\limits_{i=1}^2(b_iy+c_iz)-B'xy,-B'xz \right)$$
  has invariant hyperplanes $z=0$, $b_1y+c_1z=0$, and $b_2y+c_2z=0$. 

  The proof for $b=0$ is similar to the proof of $a = 0$. To show that the bound can be reached, consider $A'=\prod\limits_{i=1}^2(a_ix+c_iz)-C'y$, $B'=0$ for any linear homogeneous polynomial $C'$. Then the vector field
 $$\mathcal{X}=(P,Q,R)= \left( y\prod\limits_{i=1}^2(a_ix+c_iz)-C'y^2, -x\prod\limits_{i=1}^2(a_ix+c_iz)+C'xy+C'z^2, -C'yz \right)$$
 has invariant hyperplanes $z=0$, $a_1x+c_1z=0$, and $a_2x+c_2z=0$.

(2) The extactic polynomial of $\mathcal{X}$ associated with $\langle x,y \rangle$ is
$$\mathcal{E}_{\langle x,y\rangle}(\mathcal{X})=\begin{vmatrix}
        x&y\\
        P&Q
    \end{vmatrix}=-A'(x^2+y^2)+z^2(C'x-B'y).$$ Suppose $\mathcal{E}_{\langle x,y\rangle}(\mathcal{X})$ has $4$ factors of the form $ax+by$, say $a_ix+b_iy$ for $i=1,2,..,\ell \leq 4$. Then 
    \begin{equation}\label{x-y-hyperplane}
        -A'(x^2+y^2)+z^2(C'x-B'y)=p\prod\limits_{i=1}^{\ell} (a_ix+b_iy)
    \end{equation}
 for some polynomial $p\in \RR[x,y,z]$. In particular, for $z=0$, last equation corresponds that $$-A'(x,y,0)(x^2+y^2)=p(x,y,0)\prod\limits_{i=1}^{\ell} (a_ix+b_iy).$$ If $A'(x,y,0)\neq 0$ then $\ell \leq 2$. 
 Otherwise, $A'=zA''$ and $p=zp'$ for some polynomials $A'',p'\in \RR[x,y,z]$. Then \eqref{x-y-hyperplane} becomes $$-A''(x^2+y^2)+z(C'x-B'y)=p' \prod\limits_{i=1}^{\ell} (a_ix+b_iy).$$ Again, for $z=0$, this equation gives $-A''(x,y,0)(x^2+y^2)=p'(x,y,0) \prod\limits_{i=1}^{\ell} (a_ix+b_iy)$. If $A''(x,y,0)\neq 0$ then $\ell \leq 1$. Otherwise, $A''=\alpha z$ for some $\alpha \in \RR$. In that case, we obtain $\mathcal{E}_{\langle x,y\rangle}(\mathcal{X})=z^2(-\alpha (x^2+y^2)+(C'x-B'y))$. So, in this case also $\mathcal{E}_{\langle x,y\rangle}(\mathcal{X})$ has maximum two factors of the form $ax+by$. So, combining all the cases, $\mathcal{X}$ has maximum two invariant great circles of the form $\{ax+by=0\}\cap \mathbb{S}^2$.

    To prove that the bound can be reached, consider $A'=0,B'=C'=ax+by$. Then the vector field $$\mathcal{X}=(P,Q,R)=((ax+by)z^2,(ax+by)z^2, -(ax+by)(x+y)z)$$ has invariant hyperplanes $ax+by=0$ and $x-y=0$.

 (3) In \eqref{eq: 36}, if $\mathcal{X}$ is a homogeneous vector field of degree three, then $K'$ is homogeneous of degree two. Let
$$ K' = k_1x^2 + k_2y^2 + k_3z^2 + k_4xy + k_5xz + k_6yz. $$
Then, from \eqref{eq: 36}, we get the following by equating the coefficients on its both sides of each monomial.

\begin{equation*}    \label{eq: 39}
\begin{split}
x^3:~ & ba_1 = -ak_1                   \\
y^3:~ & aa_2 = bk_2                    \\
z^3:~ & ab_3 + bc_3 = ck_3             \\
x^2y:~ & aa_1 - ba_4 = bk_1 + ak_4        \\
x^2z:~ & ba_5 +cb_1 = -ck_1 - ak_5              \\
y^2z:~ & aa_6 -cc_2 = ck_2 + bk_6              \\
xy^2:~ & aa_4 -ba_2 = ak_2 + bk_4      \\
xz^2:~ &-ba_3 + ab_1 - cb_3 + bc_1 = ak_3 + ck_5       \\
yz^2:~ & aa_3 + ab_2 + bc_2 - cc_3 = bk_3 + ck_6       \\
xyz:~ & aa_5 - ba_6 - cb_2 - cc_1 = ck_4+ bk_5+ak_6.
\end{split}
\end{equation*}
Then, from the first six equations, one can compute $k_i$'s. For example $k_1=\frac{ba_1}{-a}$, $k_2=\frac{aa_2}{b}$, $k_3=\frac{ab_3+bc_3}{c}$, and so on similarly. Substituting these $k_i$'s in the remaining four equations, we get the following four equations respectively.

% \begin{equation}\label{eq: 301}
% \begin{split}
% % & \left( \frac{a^2 + b^2}{a}   \right)a_1 - \left( \frac{-a^2 + b^2}{b}   \right)a_4 = \left( \frac{a^2 + b^2}{b^2}   \right)a_2,                    \\
% E_1: &b^2a_1+a^2a_2-aba_4=0,\\
% % & \left( \frac{a^2 - c^2}{a}  \right)b_5 - ba_3 + bc_5 + \left( \frac{a^2 - c^2}{c}  \right)b_3 = \frac{ab}{c}c_3 + \frac{bc^2}{a^2}a_1 - \frac{cb}{a}a_5 - \frac{cb}{a}b_4,         \\
% E_2: &bc^3a_1+a^2bca_3-abc^2a_5+a^2(a^2+c^2)b_3-abc^2b_4-ac(a^2+c^2)b_5+a^3bc_3-a^2bcc_5=0, \\
% % & aa_3 + ab_6 + \left( \frac{ b^2 + c^2}{b}   \right)c_6 - \left( \frac{ b^2 + c^2}{c}   \right)c_3  = \frac{ab}{c}b_3 + \frac{ac}{b}a_6 + \frac{ac}{b}b_2 - \frac{ac^2}{b^2}a_2,             \\
% E_3: &ac^3a_2+ab^2ca_3-abc^2a_6-abc^2b_2-ab^3b_3+ab^2cb_6-b^2(b^2+c^2)c_3+bc(b^2+c^2)c_6=0,\\
% % & \left( \frac{a^2 + b^2}{a}   \right)a_5 + \left( \frac{a^2 + b^2}{a}   \right)b_4 - \left( \frac{a^2 + b^2}{b}   \right)a_6 - \left( \frac{a^2 + b^2}{b}   \right)b_2 -cb_6 - cc_5 = \\ 
% % & -\frac{ac}{b}c_6 - \frac{bc}{a}b_5 + \frac{b^2c}{a^2}a_1 + \frac{ac}{b}a_4 - \frac{c(b^2 + 2a^2)}{b^2}a_2\\
% E_4: &c^3(a^2+b^2)a_2+b^2c(-a^2+b^2)a_3-abc^3a_4+ab^2c^2a_5-b^3c^2a_6-b^3c^2b_2+b^3(2a^2+c^2)b_3\\
% &+ab^2c^2b_4-ab^3cb_5-b^2c(a^2+c^2)b_6+ab^2(2b^2+c^2)c_3-b^2c(b^2+c^2)c_5-ab^3cc_6=0.
% \end{split}
% \end{equation}

\begin{equation}\label{eq: 301}
\begin{split}
E_1:~ &b^2a_1+a^2a_2-aba_4=0,\\
E_2:~ &bc^3a_1+a^2bca_3-abc^2a_5-ac(a^2+c^2)b_1+a^2(a^2+c^2)b_3-a^2bcc_1+a^3bc_3=0, \\
E_3:~ &ac^3a_2+ab^2ca_3-abc^2a_6+ab^2cb_2-ab^3b_3+bc(b^2+c^2)c_2-b^2(b^2+c^2)c_3=0,\\
E_4:~ &b^2c(a^2+2b^2)a_1-a^4ca_2-ab^3ca_4-ab^2(a^2+b^2)a_5+a^2b(a^2+b^2)a_6 -ab^3cb_1 \\ &+a^2b^2cb_2+a^2b^2cc_1-a^3bcc_2=0.
\end{split}
\end{equation}
If the vector field $\mathcal{X}$ has $\{ ax + by + cz = 0   \} \cap \mathbb{S}^2$ as an invariant great circle, then \eqref{eq: 301} must be satisfied. In order for there to be additional great circles $\{ px + qy + rz = 0     \} \cap \mathbb{S}^2,$ and $\{ sx + ty + uz = 0   \} \cap \mathbb{S}^2, $ the vector field $\mathcal{X}$ must satisfy \eqref{eq: 301} with $(p,q,r)$ and $(s,t,u)$ in place of $(a,b,c)$ in addition to \eqref{eq: 301} itself. Thus we have a system of twelve equations, and solving them (with the help of SAGEMATH\footnote{ Sagemath official website:
\href{https://www.sagemath.org/}{https://www.sagemath.org/}}), we get
\begin{equation*}
% \begin{split}
% & a_1 = a_2 = a_4 = b_1 = b_3 = b_5 = c_2 = 0,         \\
% & b_2 = -a_6, ~~~~~~~~~ c_4 = a_6,                      \\
% & b_4 = rsa_5 - pua_5, ~~~~~~~~~~~~~ c_1 = pua_5 - rsa_5,  \\
% & a_3 = [ acp^2c_5 - ( r^3sa_5 + r^2a_5 - (r^2ua_5 - rc_5)p)a^2  \\ 
% &  \quad \quad  - (p^3ua_5 - (rsa_5 + a_5)p^2)c^2 ]/{ap},      \\
% & b_6 =  - [ (p^2rua_5 - (r^2sa_5 + ra_5)p)a^3cq + (p^2ua_5 + ((rsa_5 - pua_5 + a_5)q^3 \\
% & \quad \quad - (p^3ua_5 - (rsa_5 + a_5)p^2)q)c^2 - (rsa_5 + a_5)p)a^2 \\ 
% &  \quad \quad  + ((p^2rua_5 - (r^2sa_5 + ra_5)p)acq + p^2ua_5 - (rsa_5 + a_5)p)b^2 + apc_5 ]/{ap}, \\
% & c_3 =  -[ (r^3sa_5 + r^2a_5 - (r^2ua_5 - rc_5)p)a^2c + (p^3ua_5 - (rsa_5 + a_5)p^2)c^3 \\
% &  \quad \quad  + (p^2ua_5 - (rsa_5 + a_5)p)c^2 - (c^2p^2c_5 + cpc_5)a]/{a^2p},    \\
% & c_6 = -[  (p^2rua_5 - (r^2sa_5 + ra_5)p)b^3q + ((p^2rua_5 - (r^2sa_5 + ra_5)p)a^2q \\
% &  \quad \quad  + ((rsa_5 - pua_5 + a_5)q^3 - (p^3ua_5 - (rsa_5 + a_5)p^2)q)ac)b]/{ap}.  
a_1=a_2=a_4=a_5=a_6=b_1=b_3=c_2=c_3=0,\text{and } a_3=-b_2=c_1.
% \end{split}
\end{equation*}
For these coefficients, the vector field becomes $\mathcal{X}=(0,0,0)$. Hence, the maximum number of invariant great circles might be 2.

To prove that the bound can be reached, consider the following example. Suppose we have the polynomials $A'=x^2+y^2+2xy+xz+yz, B'=-y+z$, $C'=x-z$. Then 
\begin{equation*}
    \begin{split}
\mathcal{X}=&(P,Q,R)\\
                =&(A'y+B'z^2, -A'x+C'z^2, -B'xz-C'xz)\\
                =&(x^2y+y^3+2xy^2+xyz+y^2z-yz^2+z^3, -x^3-xy^2-2x^2y-x^2z-xyz+xz^2-z^3, -xz^2+yz^2)
    \end{split}
\end{equation*} is a vector field on $\mathbb{S}^2$. One can check that $x+y+z=0$ and $x+y-z=0$ are invariant hyperplanes for $\mathcal{X}$ with corresponding cofactors $-x^2+y^2$ and $-x^2+y^2-2xz+2yz$ respectively. Hence, $\{x+y+z=0\}\cap \mathbb{S}^2$ and $\{x+y-z=0\}\cap \mathbb{S}^2$ are invariant great circles of $\mathcal{X}$.
\end{proof}

Now, we shall look at invariant circles, which are not great circles.
% Theorem 1.1 of  \cite{LliPes06} tells that if a homogeneous degree 2 vector field on $\mathbb{S}^2$ has finitely many invariant circles, then every invariant circle is a great circle. But for a cubic homogeneous vector field on $\mathbb{S}^2$ with finitely many invariant circles, can have invariant circle which is not a great circle.

% \begin{theorem}\label{prop:_inv_not_gcycle}
% Assume that a homogeneous cubic vector field, $\mathcal{X},$ on $\mathbb{S}^2$ with only finitely many invariant circles has an invariant circle which is not a great circle. Then $\mathcal{X}$ has two invariant circles, which are not great circles. Further, $\mathcal{X}$ does not admit an invariant great circle.
% \end{theorem}

% \begin{remark}
% In the proof of Theorem \ref{prop:_inv_not_gcycle}, the coefficients of the cofactors are expressed in terms of the coefficients of the vector fields and the invariant circles. If the cofactor vanishes, then there will be infinitely many invariant circles.
% \end{remark}

\begin{proposition} \label{prop:inv_not_gcycle}
Assume that a cubic homogeneous vector field $\mathcal{X}=(P,Q,R)$ on $\mathbb{S}^2$ has an invariant circle which is not a great circle. Without loss of generality, assume the invariant circle to be $\{z+d=0\}\cap \mathbb{S}^2$ with $0<d<1$. Then $R=(px+qy)(-d^2(x^2+y^2+z^2)+z^2)$ for some $p,q\in \RR$.
\end{proposition}
\begin{proof}
    Recall that $\mathcal{X}=(P,Q,R)$ is a cubic homogeneous vector field on $\mathbb{S}^2$ if and only if $\mathcal{X}$ is given by \eqref{eq:cubic-hom}, i.e., $P=Ay+Bz, Q=-Ax+Cz,R=-Bx-Cy$ for some polynomials $A,B,C\in \RR[x,y,z]$.

    The equation of the cone corresponding to the circle $ \mathcal{C} : = \{z + d = 0 \} \cap \mathbb{S}^2 $ is 
$$ g_{ \mathcal{C}} : = -d^2(x^2+y^2+z^2)+z^2. $$
Therefore, the condition for the circle $\mathcal{C}$ to be invariant is 
\begin{equation}   \label{eq: 302}
\mathcal{X}g_{\mathcal{C}} = K_{\mathcal{C}} g_{\mathcal{C}}, \end{equation}
where the polynomial $K_{\mathcal{C}}$ is the cofactor. Note that $K_{\mathcal{C}}$ is a homogeneous polynomial. Now,
\begin{equation*}
    \begin{split}
        &\mathcal{X}g_{\mathcal{C}}=K_{\mathcal{C}}g_{\mathcal{C}}\\
       \implies& -2d^2(Px+Qy+Rz)+2Rz=K_{\mathcal{C}}g_{\mathcal{C}}\\
       \implies& 2Rz=K_{\mathcal{C}}g_{\mathcal{C}}\quad \text{since $Px+Qy+Rz=0$}.
    \end{split}
\end{equation*}
Since $d\neq 0$, $z$ does not divide $g_{\mathcal{C}}$. So, the last implication tells us that $z$ divides $K_{\mathcal{C}}$, assume that $K_{\mathcal{C}}=2z(px+qy+rz)$ for some $p,q,r\in \RR$. Hence, $R=(px+qy+rz)(-d^2(x^2+y^2+z^2)+z^2)$. Since $R=-Bx-Cy$, we get $r(1-d^2)=0$. So, either $r=0$ or $d=\pm 1$. Since $d<1$, $r$ must be zero. Then $R=(px+qy)(-d^2(x^2+y^2+z^2)+z^2)$. 
%Note that $\{z-d=0\}\cap \mathbb{S}^2$ is also an invariant circle for this vector field.
\end{proof}
\begin{remark}
    Proposition \ref{prop:inv_not_gcycle} can also be proved for degree $n (> 3)$ homogeneous vector fields on $\mathbb{S}^2$. In that case, $R=(Mx+Ny)(-d^2(x^2+y^2+z^2)+z^2)$ for some $M,N\in \RR[x,y,z]$ such that $\deg M = (n-3)$ or $\deg N = (n-3)$.
\end{remark}

\section{Cubic Kolmogorov vector fields on $\mathbb{S}^2$}\label{sec:cubic-kolm}

We recall that the form of a cubic Kolmogorov vector field on $\mathbb{S}^2$ has been examined in Corollary \ref{cr:kolm3}. 
This section aims to study the dynamical properties of these cubic Kolmogorov vector fields on $\mathbb{S}^2.$  For convenience, we shall rewrite the general form of a cubic Kolmogorov vector field, $\mathcal{X} = (P, Q, R)$ defined on $\mathbb{S}^2,$

\begin{equation}    \label{eq:kolmsys}
\begin{split}
P &= x(\alpha (1-x^2-y^2-z^2) + Ay^2 + Bz^2), \\
Q &= y(\beta (1-x^2-y^2-z^2) - Ax^2 + Cz^2),  ~\mbox{and} \\
R &= z(\gamma (1-x^2-y^2-z^2) - Bx^2 - Cy^2),
\end{split}    
\end{equation}
where $\alpha, \beta, \gamma$ and $A, B, C$ are constants. In what follows, we shall assume that $A,B,C \neq 0$ unless it is specified. Notice from \eqref{eq:kolmsys} that the hyperplanes $\{ x = 0    \}$, $\{ y = 0    \}$, and $\{ z = 0    \}$ are invariant with respect to the flow of $\mathcal{X}.$ This gives that on $\mathbb{S}^2,$ the great circles determined by the intersection of these hyperplanes with $\mathbb{S}^2,$ that is, $\{ x = 0    \} \cap \mathbb{S}^2$, $\{ y = 0    \} \cap \mathbb{S}^2$, and $\{ z = 0    \} \cap \mathbb{S}^2$ are invariant with respect to the flow of $\mathcal{X}.$
From this, it implies that the points where these great circles intersect are singular points. Hence $(\pm 1, 0, 0),$ $(0, \pm 1, 0),$ and $(0, 0, \pm 1)$ are singular points of $\mathcal{X}.$ We want to determine if $\mathcal{X}$ has any other singular points on $\mathbb{S}^2.$ Note that from \eqref{eq:kolmsys}, additional singular points can be obtained by solving the following equations. 
\begin{equation*}
    \begin{split}
        Ay^2 + Bz^2 = 0, \\
        -Ax^2 + Cz^2 = 0, \\
        ~~-Bx^2 - Cy^2 = 0.
    \end{split}
\end{equation*}
These equations, together with the fact that we are on the sphere, $x^2 + y^2 + z^2 = 1,$ give
\begin{equation}       \label{eq:extra_singularities}
%\begin{aligned}
 x = \pm \sqrt{ \frac{-C}{B-A-C}}, \quad   
 y = \pm \sqrt{ \frac{B}{B-A-C}}, \quad \mbox{and} \quad z = \pm \sqrt{ \frac{-A}{B-A-C}}.
%\end{aligned}
\end{equation}

Observe that $(x,y,z)$ is a point in $\mathbb{R}^3$ when $A,C > 0,$ $B < 0$ or when $A,C < 0,$ $B > 0.$ If neither of these conditions are true, $\mathcal{X}$ has only $(\pm 1, 0, 0),$ $(0, \pm 1, 0),$ and $(0, 0, \pm 1)$ as singularities.

 Our objective is to determine the phase portraits of cubic Kolmogorov vector fields on $\mathbb{S}^2.$ From \eqref{eq:kolmsys}, we notice that the antipodal transformation $(x, y, z) \mapsto (-x, -y, -z)$ on $\RR^3$ causes $(P, Q, R) \mapsto (-P, -Q, -R)$. This implies that trajectories in the southern hemisphere mirror trajectories in the northern hemisphere with the direction of flow reversed. Thus, in order to determine the phase portrait of these vector fields on $\mathbb{S}^2$, it is sufficient to determine the phase portrait on one of the hemispheres (including its boundary). In this section, we shall use the stereographic projection $\Phi$ of Section \ref{sec:prelim} from the South pole,  to determine the phase portrait on the northern hemisphere. The northern hemisphere along with its boundary is mapped to the closed unit disk $D^2 := \{ (u,v) \in \mathbb{R}^2 ~ | ~ u^2 + v^2 \leq 1    \}$ by $\Phi.$ 
We find the converted vector field in $\RR^2$ of the cubic Kolmogorov vector field on $\mathbb{S}^2$ given by \eqref{eq:kolmsys}. For that, we have the following:

% \begin{equation*}
% \begin{split}
%  \Tilde{P} = \alpha (2u)(u^2 + v^2 + 1)^2 - \alpha (2u)^3 +  (A - \alpha)(2u)(2v)^2 + (B - \alpha)(2u)(1 - u^2 - v^2)^2,   \\
%  \Tilde{Q} = \beta (2v)(u^2 + v^2 + 1)^2 - \beta (2v)^3 -  (A + \beta)(2v)(2u)^2 + (C - \beta)(2v)(1 - u^2 - v^2)^2,   \\
% \Tilde{R} = \gamma (1 - u^2 - v^2)(u^2 + v^2 + 1)^2 - \gamma (1 - u^2 - v^2)^3 -  (B + \gamma)(1 - u^2 - v^2)(2u)^2 - \\ (C + \gamma)(2v)^2(1 - u^2 - v^2).
% \end{split}
% \end{equation*}
\begin{equation*}
\begin{split}
 \Tilde{P} &= 2u(4Av^2+B(1-u^2-v^2)^2),  \\
 \Tilde{Q} &= 2v(-4Au^2+C(1-u^2-v^2)^2),  ~\mbox{and}\\
\Tilde{R} &= -4(1-u^2-v^2)(Bu^2+Cv^2).
\end{split}
\end{equation*}
This gives
% \begin{equation}      \label{eq:kolm_stereop}
%     \begin{split}
%       &  \mathcal{P} = \dot{u} = \alpha [2u(u^2 + v^2 + 1) + (2u)^3] +   (2u)(2v)^2(A - \alpha) + (2u)(1 - u^2 - v^2)^2(B - \alpha) - u \Tilde{R}, \\
%       &  \mathcal{Q} = \dot{v} = \beta [(2v)(u^2 + v^2 + 1)^2 - (2v)^3] - (2v)(2u)^2(A + \beta) + (C - \beta)(2v)(1 - u^2 - v^2)^2 -v \Tilde{R}.
%     \end{split}
% \end{equation}

\begin{equation}      \label{eq:kolm_stereop}
    \begin{split}
      &  \mathcal{P} = \dot{u} = 2u(4Av^2+B(1-u^2-v^2)^2) - u \Tilde{R}, ~\mbox{and}\\
      &  \mathcal{Q} = \dot{v} = 2v(-4Au^2+C(1-u^2-v^2)^2)-v \Tilde{R}.
    \end{split}
\end{equation}

We shall now study the behaviour of this system near the singular points $(1, 0, 0),$ $(0, 1, 0),$ and $(0, 0, 1).$ Note that the behaviour of the system near the singular points $(-1, 0, 0),$ $(0, -1, 0),$ and $(0, 0, -1)$ is identical (with reverse arrows) to the behaviour around their antipodal counterparts by the discussion above.

 In the case(s) where these are the only singular points of $\mathcal{X},$ we shall be able to draw phase portraits. We shall compute the Jacobian, $J$ at singular points of the vector field $\Phi_*(\mathcal{X}) = (\mathcal{P}, \mathcal{Q})$ in $\mathbb{R}^2.$ We have
 
\begin{equation}    \label{eq:jacobian}
  J =  \begin{pmatrix}
\mathcal{P}_u &  \mathcal{P}_v \\ \mathcal{Q}_u &  \mathcal{Q}_v\end{pmatrix}. 
\end{equation}

From \eqref{eq:kolm_stereop}, we have
\begin{equation*}
    \begin{split}
        % \mathcal{P}_u = \alpha[2(u^2 + v^2 + 1)^2 + 8u^2(u^2 + v^2 + 1) - 6(2u)^2] +  2(2v)^2(A - \alpha) + \\ (B - \alpha)[2(1 - u^2 - v^2)^2 - 8u^2(1 - u^2 - v^2)] - \Tilde{R} - u \Tilde{R}_u,
&\mathcal{P}_u =2(4Av^2+B(1-u^2-v^2)^2)-8Bu^2(1-u^2-v^2)-\Tilde{R}-u\Tilde{R}_u,\\
&\mathcal{P}_v=2u(8Av-4Bv(1-u^2-v^2))-u\Tilde{R}_v,\\
&\mathcal{Q}_u=2v(-8Au-4Cu(1-u^2-v^2))-v\Tilde{R}_u,\\
&\mathcal{Q}_v=2(-4Au^2+C(1-u^2-v^2)^2)-8Cv^2(1-u^2-v^2)-\Tilde{R}-v\Tilde{R}_v.
    \end{split}
\end{equation*}

% \begin{equation*}
%     \begin{split}
%         % \mathcal{P}_v = 16uv(A - \alpha) - 8uv(1 - u^2 - v^2)(B - \alpha) + 8uv(u^2 + v^2 + 1)\alpha - \\ u \Tilde{R}_v,
% \mathcal{P}_v=2u(8Av-4Bv(1-u^2-v^2))-u\Tilde{R}_v,
%     \end{split}
% \end{equation*}

% \begin{equation*}
% \begin{split}
%     % \mathcal{Q}_u = -16uv(A + \beta) - 8uv(1 - u^2 - v^2)(C - \beta) + 8uv(u^2 + v^2 + 1)\beta - \\ v \Tilde{R}_u,
% \mathcal{Q}_u=2v(-8Au-4Cu(1-u^2-v^2))-v\Tilde{R}_u,
%     \end{split}
% \end{equation*}

% \begin{equation*}
%     \begin{split}
%         % \mathcal{Q}_v = \beta [2(u^2 + v^2 + 1)^2 + 8v^2(1 + u^2 + v^2) - 6(2v)^2] - 2(2u)^2(A + \beta) + \\ (C - \beta)[2(1 - u^2 - v^2)^2 - 8v^2(1 - u^2 - v^2)] - \Tilde{R} - v \Tilde{R}_v,
% \mathcal{Q}_v=2(-4Au^2+C(1-u^2-v^2)^2)-8Cv^2(1-u^2-v^2)-\Tilde{R}-v\Tilde{R}_v.
%     \end{split}
% \end{equation*}
Also,
\begin{equation*}
    \begin{split}
        % \Tilde{R}_u = \gamma (u^2 + v^2 + 1)[(-2u)(u^2 + v^2 + 1) + (1 - u^2 - v^2)(4u)] + 6u \gamma (1 - u^2 - v^2)^2 - \\ (B + \gamma)[-8u^3 + 8u(1 - u^2 - v^2)] + (C + \gamma)(2v)^2(2u),
&\Tilde{R}_u=8u(Bu^2+Cv^2)-8Bu(1-u^2-v^2),\\
&\Tilde{R}_v=8v(Bu^2+Cv^2)-8Cv(1-u^2-v^2).
    \end{split}
\end{equation*}

% \begin{equation*}
%     \begin{split}
%          % \Tilde{R}_v = \gamma (u^2 + v^2 + 1)[(-2v)(u^2 + v^2 + 1) + (1 - u^2 - v^2)(4v)] + 6v \gamma (1 - u^2 - v^2)^2 - \\ (C + \gamma)[-8v^3 + 8v(1 - u^2 - v^2)] + (B + \gamma)(2u)^2(2v).
% \Tilde{R}_v=8v(Bu^2+Cv^2)-8Cv(1-u^2-v^2).
%     \end{split}
% \end{equation*}

Now we compute the Jacobian for the singular points $(1, 0, 0),$ $(0, 1, 0),$ and $(0, 0, 1).$

\begin{enumerate}
    \item For $(1, 0, 0)$: The stereographic projection $\Phi$ maps this point to $(1, 0)$ on the plane. We compute
    $$\mathcal{P}_u(1, 0) = -8B, \quad \mathcal{P}_v(1, 0) = 0, \quad \mathcal{Q}_u(1, 0) = 0, \quad \mathcal{Q}_v(1, 0) = -8A.$$
    Therefore,
    \begin{equation*}
     J_{(1, 0)} =  8 \begin{pmatrix}
    -B &  0 \\ 0 &  -A\end{pmatrix}. 
    \end{equation*}

    \item For $(0, 1, 0)$: The stereographic projection $\Phi$ maps this point to $(0, 1)$ on the plane. We compute
    $$\mathcal{P}_u(0, 1) = 8A, \quad \mathcal{P}_v(0, 1) = 0, \quad \mathcal{Q}_u(0, 1) = 0, \quad \mathcal{Q}_v(0, 1) = -8C.$$
    Therefore,
    \begin{equation*}
     J_{(0, 1)} =  8 \begin{pmatrix}
    A &  0 \\ 0 &  -C\end{pmatrix}. 
    \end{equation*}

    \item For $(0, 0, 1)$: The stereographic projection $\Phi$ maps this point to $(0, 0)$ on the plane. We compute
    $$\mathcal{P}_u(0, 0) = 2B, \quad \mathcal{P}_v(0, 0) = 0, \quad \mathcal{Q}_u(0, 0) = 0, \quad \mathcal{Q}_v(0, 0) = 2C.$$
    Therefore,
    \begin{equation*}
     J_{(0, 0)} =  2 \begin{pmatrix}
    B &  0 \\ 0 &  C\end{pmatrix}. 
    \end{equation*}
\end{enumerate}

Assume that $A, B, C \neq 0.$ Let
\begin{equation}\label{eq:a_and_b}
(a)~~ A,C > 0, B < 0, \quad \mbox{and} \quad
    (b)~~ A, C < 0, B > 0. 
\end{equation}
Thus, if neither (a) nor (b) is true, then $(\pm 1, 0, 0),$ $(0, \pm 1, 0),$ and $(0, 0, \pm 1)$ are the only singular points of $\mathcal{X}$ on $\mathbb{S}^2$, and we draw the phase portraits for each of these cases in Figure \ref{fig:all}.

% \begin{figure}[H]
%     \centering
%     \includegraphics[scale=0.7]{Fig1.png}
%     \caption{When $A,B,C > 0$.}
%     \label{fig:fig1}
% \end{figure}

% \begin{figure}[H]
%     \centering
%     \includegraphics[scale=0.7]{Fig2.png}
%     \caption{When $A<0,$ $B,C > 0$.}
%     \label{fig:fig2}
% \end{figure}

% \begin{figure}[H]
%     \centering
%     \includegraphics[scale=0.7]{Fig3.png}
%     \caption{When $A,B>0,$ $C < 0$.}
%     \label{fig:fig3}
% \end{figure}

\begin{figure}
    \centering
    \begin{subfigure}{0.32\textwidth}
        \includegraphics[width=\linewidth]{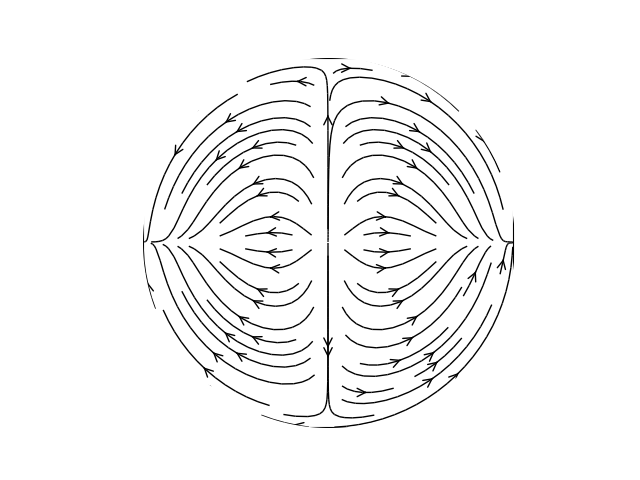}
        \caption {When $A,B,C > 0$.}
        \label{fig:fig1}
    \end{subfigure}
    \hfill
    \begin{subfigure}{0.32\textwidth}
        \includegraphics[width=\linewidth]{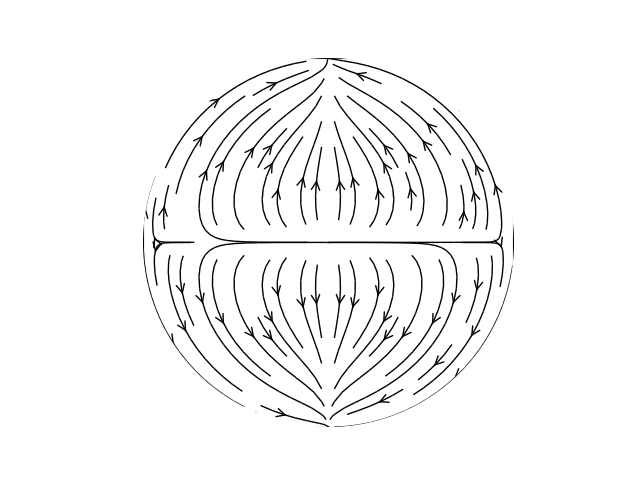}
        \caption{When $A<0,$ $B,C > 0$.}
        \label{fig:fig2}
    \end{subfigure}
    \hfill
    \begin{subfigure}{0.32\textwidth}
        \includegraphics[width=\linewidth]{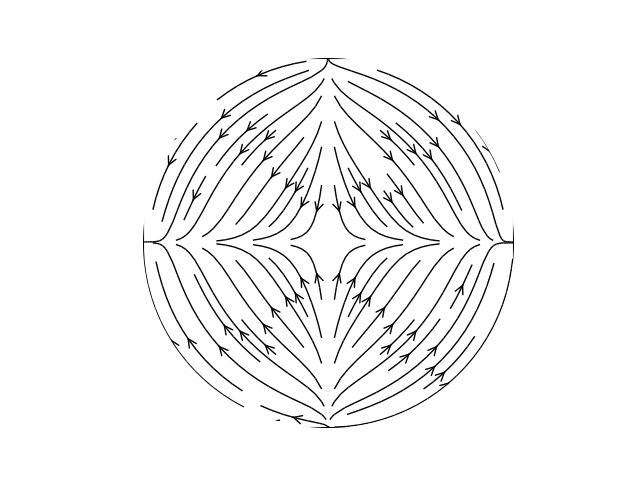}
        \caption{When $A,B>0,$ $C < 0$.}
        \label{fig:fig3}
    \end{subfigure}
    % \caption{Phase portraits when neither (a) nor (b) of \eqref{eq:a_and_b} is true.}
    % \label{fig:all}

       % \centering
       \hfill
    \begin{subfigure}{0.32\textwidth}
        \includegraphics[width=\linewidth]{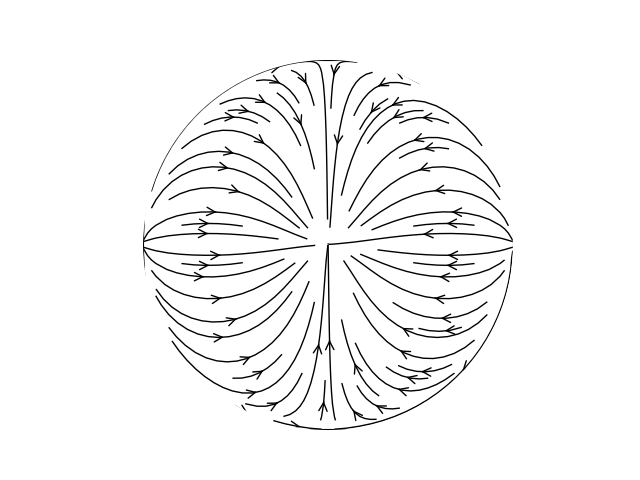}
        \caption {When $A,B,C < 0$.}
        \label{fig:fig1a}
    \end{subfigure}
    \hfill
    \begin{subfigure}{0.32\textwidth}
        \includegraphics[width=\linewidth]{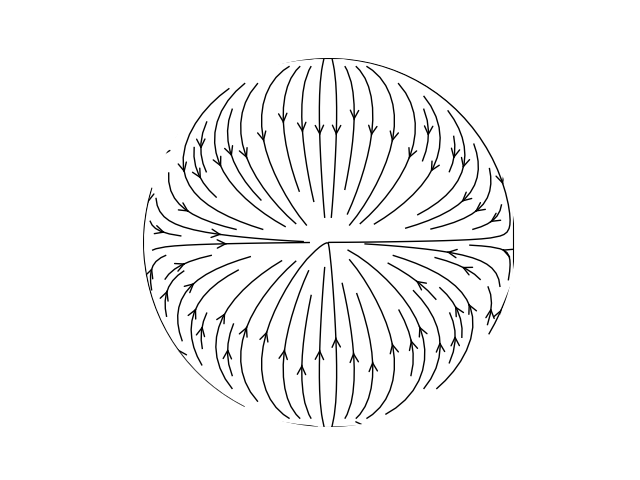}
        \caption{When $A>0,$ $B,C < 0$.}
        \label{fig:fig2a}
    \end{subfigure}
    \hfill
    \begin{subfigure}{0.32\textwidth}
        \includegraphics[width=\linewidth]{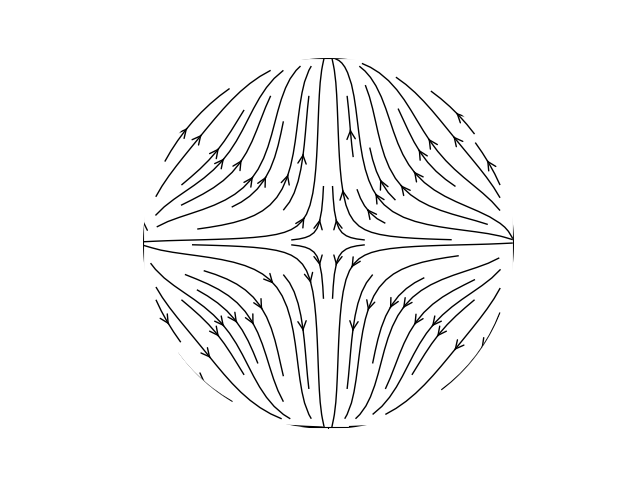}
        \caption{When $A,B<0,$ $C > 0$.}
        \label{fig:fig3a}
    \end{subfigure}
    \caption{Phase portraits when neither (a) nor (b) of \eqref{eq:a_and_b} is true.}
    \label{fig:all}
\end{figure}

\begin{remark}
%\begin{enumerate}
 %   \item A uniform change of signs leads to the same phase portrait but with all arrows reversed. For example, when $A, B, C < 0,$ the phase portrait is the same as in        but with all arrows pointing in the opposite direction.

 When one of $A, B, C$ is zero, then the singular points are no longer isolated. For example, we plot the phase portrait when $A = 0$ and $B, C > 0$ in  Figure \ref{fig:fig4}. In this case, all points on the boundary of the disk (i.e., the unit circle) are singular.
%\end{enumerate}
\end{remark}

\begin{figure}[H]
    \centering
    \includegraphics[scale=0.32]{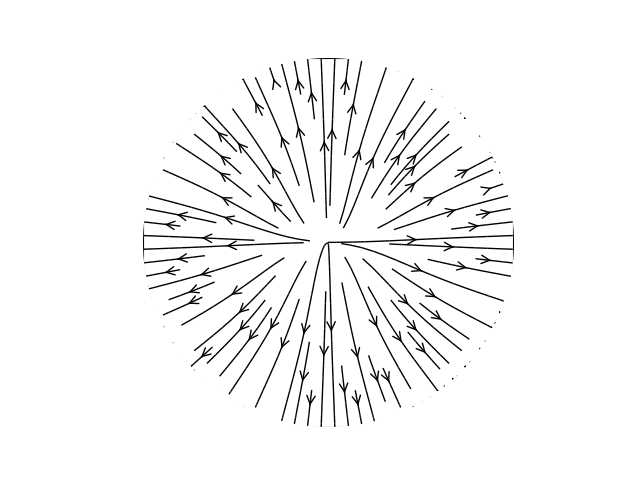}
    \caption{}
    \label{fig:fig4}
\end{figure}

From the above discussion, we see that when neither (a) nor (b) is true in \eqref{eq:a_and_b}, any Kolmogorov vector field of the form in \eqref{eq:kolmsys} has phase portrait that is topologically equivalent to either (I) or (III) in Figure \ref{fig:all}. So we have the following theorem.

\begin{theorem}   \label{thm:periodic_orbit_kolm}
    A cubic Kolmogorov vector field on $\mathbb{S}^2$ of the form given in \eqref{eq:kolmsys} admits no periodic orbit if $A,B,C$ in \eqref{eq:kolmsys} do not satisfy the following.
$$(1)~A,C > 0, B < 0,\quad \mbox{and}\quad (2)~A,C < 0, B > 0.$$
\end{theorem}

In the remaining, we discuss what happens when condition (a) in \eqref{eq:a_and_b} is true. Note that the same analysis will also apply for (b), albeit with a change in the sense of the orbits in the phase portrait. That is when $A,C > 0$ and $B<0.$ We know that in this case, there exist extra singularities given in \eqref{eq:extra_singularities}. Applying the stereographic projection to these singular points, we have,
$$ u = \frac{\pm \sqrt{-C} }{ \sqrt{B-A-C} \pm \sqrt{-A} } \quad , \quad   v = \frac{\pm \sqrt{B} }{ \sqrt{B-A-C} \pm \sqrt{-A} }.         $$
Let us fix a singular point
$$ u_0 = \frac{ \sqrt{-C} }{ \sqrt{B-A-C} + \sqrt{-A} } \quad , \quad   v_0 = \frac{ \sqrt{B} }{ \sqrt{B-A-C} + \sqrt{-A} }.         $$
We note that the analysis below can be carried out entirely similarly for the other singular points. We need to evaluate the Jacobian in \eqref{eq:jacobian} at $(u_0, v_0).$ After the calculation, we get the following.

\begin{equation*}
    \begin{split}
\mathcal{P}_u(u_0,v_0) &= \frac{8AB}{D^2} + 2B \left( \frac{D^2 + C - B}{D^2} \right)^2,      \\
\mathcal{P}_v(u_0,v_0) &= \frac{16 A \sqrt{-BC}}{D^2} + 8 \frac{(C-B) \sqrt{-BC}}{D^2} \left( \frac{D^2 + C - B}{D^2} \right),    \\
\mathcal{Q}_u(u_0,v_0) &= \frac{-16 A \sqrt{-BC}}{D^2} - 8 \frac{(C-B) \sqrt{-BC}}{D^2} \left( \frac{D^2 + C - B}{D^2} \right),    \\
\mathcal{Q}_v(u_0,v_0) &= \frac{8AC}{D^2} + 2C \left( \frac{D^2 + C - B}{D^2} \right)^2,
\end{split}
\end{equation*}
where $D := \sqrt{B-A-C} + \sqrt{-A}.$ Notice that with our choice of the coefficients $A,B$ and $C$, the term $D$ is imaginary and so $D^2 < 0.$ We take 
$$F = \frac{8A}{D^2} + 2 \left( \frac{D^2 + C - B}{D^2} \right)^2,$$
then the characteristic polynomial $c(\lambda)$ for the Jacobian matrix  $J_{(u_0,v_0)}$ is 
\begin{equation}     \label{eq:characteristic_equation}
c(\lambda) = \lambda^2 - (B+C)F \lambda + BCF^2 + \mathcal{P}_v^2(u_0,v_0).         
\end{equation}
The roots of this polynomial are by definition the eigenvalues of the Jacobian, and from these eigenvalues, we shall determine the nature of the singular point.
The discriminant, $\Delta$, of $c(\lambda)$ is  
\begin{equation*}
    \begin{split}
\Delta := & (B + C)^2F^2 - 4BCF^2 - 4 \mathcal{P}_v^2(u_0,v_0)    \\
        = & (C - B)^2F^2 - 4 \mathcal{P}_v^2(u_0,v_0)     \\
        = & ((C-B)F + 2 \mathcal{P}_v(u_0,v_0))((C - B)F - 2 \mathcal{P}_v(u_0,v_0)).
\end{split}
\end{equation*}
Now,

\begin{equation*}
    \begin{split}
&(C-B)F + 2 \mathcal{P}_v(u_0,v_0)\\
=& \frac{8A(C - B)}{D^2} + 2 (C - B)\left( \frac{D^2 + C - B}{D^2} \right)^2 + \frac{32 A \sqrt{-BC}}{D^2} +  \frac{16(C-B) \sqrt{-BC}}{D^2} \left( \frac{D^2 + C - B}{D^2} \right)     \\
 =& \frac{8A(C - B + 4 \sqrt{-BC})  }{D^2} + \frac{2(C - B) (D^2 + C - B + 8 \sqrt{-BC})}{D^2}  \left( \frac{D^2 + C - B}{D^2} \right)      .
\end{split}
\end{equation*}
For large enough $A,$ the expression $(C - B)F + 2\mathcal{P}_v(u_0,v_0)$ is negative. In fact
\begin{equation*}
    \begin{split}
&\lim_{A \rightarrow \infty}   \{  (C - B)F + 2\mathcal{P}_v(u_0,v_0) \}  \\ 
 =& \lim_{A \rightarrow \infty} \left( \frac{8A[(\sqrt{C} + \sqrt{-B})^2 + 2 \sqrt{-BC}]}{D^2} + \frac{2(C - B)[D^2 + (\sqrt{C} + \sqrt{-B})^2 + 6 \sqrt{-BC}]}{D^2}  \left( \frac{D^2 + C - B}{D^2} \right) \right) \\
 =& \lim_{A \rightarrow \infty} \left( \frac{8A[( C - B + 4 \sqrt{-BC}]}{B - 2A - C + 2\sqrt{-A}\sqrt{B-A-C} } + \frac{2(C - B)[ -2A + 2\sqrt{-A} \sqrt{B - A - C} + 8 \sqrt{-BC}]}{B - 2A - C + 2\sqrt{-A}\sqrt{B-A-C}}  \left( \frac{-2A + 2\sqrt{-A} \sqrt{B - A - C} }{B - 2A - C + 2\sqrt{-A}\sqrt{B-A-C}} \right) \right)  \\
=& -8\sqrt{-BC}.
\end{split}
\end{equation*}
In a similar manner
\begin{equation*}
    \begin{split}
&(C-B)F - 2 \mathcal{P}_v(u_0,v_0)\\
=& \frac{8A(C - B)}{D^2} + 2 (C - B)\left( \frac{D^2 + C - B}{D^2} \right)^2 - \frac{32 A \sqrt{-BC}}{D^2} -  \frac{16(C-B) \sqrt{-BC}}{D^2} \left( \frac{D^2 + C - B}{D^2} \right)     \\
 =& \frac{8A(C - B - 4 \sqrt{-BC})} {D^2} + \frac{2(C - B) (D^2 + C - B - 8 \sqrt{-BC})}{D^2}  \left( \frac{D^2 + C - B}{D^2} \right)  .
\end{split}
\end{equation*}
For large enough $A,$ the expression $(C - B)F - 2 \mathcal{P}_v$ is positive. In fact
$$ \lim_{A \rightarrow \infty}   \{  (C - B)F - 2\mathcal{P}_v \} = 8 \sqrt{-BC}. $$
From the above, we conclude that for $A$ large enough, the discriminant $\Delta$ of $c(\lambda)$ is negative and hence the Jacobian matrix at $(u_0,v_0)$ has a pair of complex conjugate eigenvalues, say $\mu \pm i \nu$. Thus we get the following result using \cite[Theorem 4, Page-143]{Per01} and \cite[Corollary to Theorem 5, Page-145]{Per01}.
\begin{theorem}\label{thm:center-focus}
    For $A$ large enough, the singular point $(u_0, v_0)$ is either a center or a focus.
\end{theorem}

\begin{remark}
    If the real part of the eigenvalue $\mu > 0,$ then the singular point $(u_0, v_0)$ is an unstable focus and if $\mu < 0,$ then the singular point $(u_0, v_0)$ is a stable focus, see \cite[Theorem 4, Page-143]{Per01}.
\end{remark}

\begin{example}
We compute the discriminant and also determine the nature of the extra singularities when $A = 5,$ $B = -1$ and $C = 2.$ In this case, the computation gives that $(B + C)F = -0.0001$ and $\Delta = -123.54.$

 From \eqref{eq:characteristic_equation}, we see that the eigenvalues are given by 
 $$ \frac{(B + C)F \pm \sqrt{\Delta}}{2}.         $$
Therefore, the eigenvalues are complex conjugate with a negative real part. This implies that the singular point is a stable focus. A similar calculation will also show that the remaining singularities belonging to the interior of $D^2$ are also stable foci. In Figure   \ref{fig:fig5}, we plot the phase portrait of this case.
\end{example}

\begin{figure}[H]
    \centering
    \includegraphics[scale=0.32]{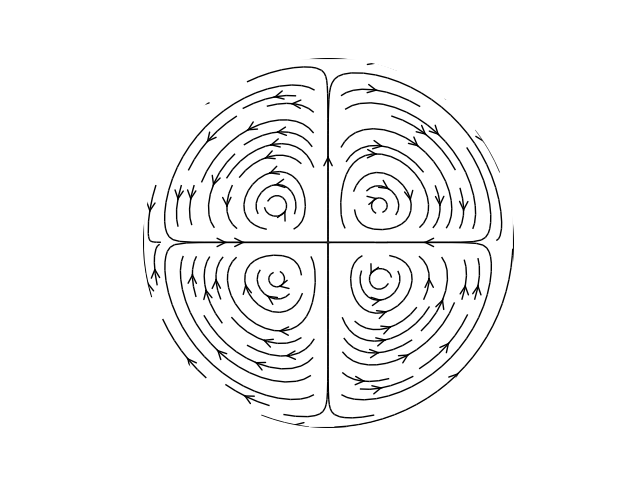}
    \caption{When $A = 5$ and $B = -1$, and $C = 2$.}
    \label{fig:fig5}
\end{figure}

\noindent {\bf Acknowledgments.}
The first author was supported by a Senior Research Fellowship of the University Grants Commission of India for the duration of this work. The second
author is supported by the Prime Minister’s Research Fellowship, Government of India. The third author thanks `ICSR office IIT Madras' for SEED research grant.

\bibliography{Ref_InvariantAlgebraicSets.bib}
\bibliographystyle{abbrv}

\end{document}